\documentclass[12pt,reqno]{amsart}
\usepackage[utf8]{inputenc}
\usepackage[margin=1in]{geometry}
\usepackage{amsmath}
\usepackage{amssymb,latexsym}
\usepackage{enumerate}
\usepackage{cite}
\usepackage{xcolor}

\newtheorem{thm}{Theorem}[section]
\newtheorem{cor}[thm]{Corollary}
\newtheorem{lem}[thm]{Lemma}
\newtheorem{prop}[thm]{Proposition}

\theoremstyle{definition}

\newtheorem*{hypothesis-s}{Hypothesis S}

\theoremstyle{remark}
\newtheorem{rem}{Remark}[section]

\numberwithin{equation}{section}

\def \P  {{\mathcal P}}

\def \NN {\mathbb N}
\def \CC {\mathbb C}
\def \QQ {\mathbb Q}
\def \RR {\mathbb R}
\def \ZZ {\mathbb Z}

\def \F  {{\mathcal F}}

\def \I  {{\mathcal I}}
\def \J  {{\mathcal J}}

\def \P  {{\mathcal P}}

\def \M  {{\mathcal M}}

\def \d {\text{d}}
\renewcommand{\Re}{\mathop{\mathrm{Re}}}

\newcommand{\mop}{(\text{mod}\,p)}

\newcommand{\mopq}{(\text{mod}\,pq)}
\newcommand{\moq}{({\rm mod}\,q)}
\newcommand{\modq}{({\rm mod}\,q_1)}

\newcommand{\eval}[2][\right]{\relax
  \ifx#1\right\relax \left.\fi#2#1\rvert}

\let\abs=\envert

\begin{document}
\vskip 5mm

\title[]{\textbf {Correlations of multiplicative functions with automorphic $L$-functions}}

\author{Yujiao Jiang}

\address{Yujiao Jiang\\
School of Mathematics and Statistics
\\
Shandong University
\\
Weihai
\\
Shandong 264209
\\
China}
\email{yujiaoj@sdu.edu.cn}

\author{Guangshi L\"u}

\address{Guangshi L\"u\\
School of Mathematics
\\
Shandong University
\\
Jinan
\\
Shandong 250100
\\
China}
\email{gslv@sdu.edu.cn}

\date{\today}

\begin{abstract}
\small {
Let $\lambda_{\phi}(n)$ be the Fourier coefficients of a Hecke holomorphic or Hecke--Maass cusp form on ${\rm SL}_2(\mathbb Z)$, and $f$ be any multiplicative function that satisfies two mild hypotheses. We establish a non-trivial upper bound for the correlation $\sum_{n \leq X}f(n)\lambda_{\phi}(n+h)$ uniformly in $0<|h|\ll X$. As applications, we consider some special cases, including $\lambda_{\pi}(n), \,\mu(n)\lambda_{\pi}(n)$ and any divisor-bounded multiplicative function. Here $\lambda_{\pi}(n)$ denotes the $n$-th Dirichlet coefficient of $\text{GL}_m$ automorphic $L$-function $L(s,\pi)$ for an automorphic irreducible cuspidal representation $\pi$, and $\mu(n)$ denotes the M\"obius function. In particular, some savings are achieved for shifted convolution problems on ${\rm GL}_m\times {\rm GL}_2\, (m\geq 4)$ and Hypothesis C for the first time.
}
\end{abstract}
\maketitle

\section{Introduction}
Let $f(n)$ and $g(n)$ denote two multiplicative functions. It is a classical and fundamental problem in analytic number theory to understand the correlations
\begin{equation}\label{eq-cor}
\sum_{n\leq X}f(n)g(n+h),
\end{equation}
where $h$ is a non-zero integer. It is directly related to many celebrated unsolved problems in number theory, such as the Chowla conjecture and the divisor correlation conjecture. If $h=0$, this problem then degenerates to the mean values of the multiplicative functions, which can be tackled by a coherent set of tools. However, very little is known about correlations \eqref{eq-cor}. The difficulty here is that the shift parameter $h$ destroys the multiplicativity. It is worth noting that Klurman \cite{Klurman-2017} recently obtained an asymptotic formula of correlation \eqref{eq-cor} for some bounded pretentious multiplicative functions $f$ and $g$.

In this present work, we focus on the special case $g(n)=\lambda_{\phi}(n)$, where $\lambda_{\phi}(n)$ is the Fourier coefficient of any ${\rm GL}_2$ Hecke cusp form. This coefficient has been studied intensively for arithmetical and analytical reasons. Our goal is to obtain a slightly general result on correlations of the coefficient $\lambda_{\phi}(n)$ with a class of multiplicative functions.
Let $\F$ denote the class of all multiplicative functions  $f$ with the following weak hypotheses:

\smallskip
(i) There exists some constant $c\geq 0$ such that
\begin{equation*}
\sum_{n \leq X}|f(n)|^2 \ll X(\log X) ^{c}.
\end{equation*}

\smallskip

(ii)
Let $P$ denote the product of primes $p$ which belong to the interval $[(\log X)^{16(c+2)},X^{\delta}]$. There exists some constant $\gamma\geq 0$ such that
\begin{equation*}
\sum_{\substack{n \leq X\\ \left(n, P\right)=1}}|f(n)|^2 \ll X\frac{(\log \log X)^\gamma}{\log X},
\end{equation*}
where the implied constant depends on $c,\delta$. Here the constant $c$ is the same as in Hypothesis (i), and the parameter $\delta$ is some constant with $0<\delta<1/22$.

\begin{thm}\label{thm-main}
	Let $f\in \F$, and let $\lambda_{\phi}(n)$ be the Fourier coefficients of a  Hecke holomorphic or Hecke--Maass cusp form on ${\rm SL}_2(\ZZ)$. For any non-zero $h\ll X$, we have
	$$
	\sum_{n \leq X}f(n)\lambda_{\phi}(n+h)\ll X\frac{(\log \log X)^\frac{\gamma+1}{2}}{\log X},
	$$
	where the implied constant depends on $\delta$, the form $\phi$ and these implied constants in Hypotheses (i) and (ii).
\end{thm}

\vskip 5mm
Finally, we provide a brief overview of the proof of Theorem \ref{thm-main}. Recall that Bourgain, Sarnak and Ziegler \cite{BSZ-2013} formulated a finite version of Vinogradov's bilinear technique, which also appeared in an earlier paper of K\'atai. This is the so-called Bourgain--K\'atai--Sarnak--Ziegler (BKSZ for short) criterion, which is very effective and widely applied to attack Sarnak's disjointness conjecture. Our work is primarily based on the generalized BKSZ criterion. For our purpose, we have to generalize the BKSZ criterion to a greater extent, when compared with the previous works \cite{CafPerZac20, JL2019} of Cafferata--Perelli--Zaccagnini and Jiang--L\"u.

From the idea of BKSZ at the beginning, the sum of what we are concerned about
\[
\sum_{n\leq X} f(n) \lambda_{\phi}(n+h)
\]
can be rearranged as a sifted sum
\begin{equation}\label{eq-seive}
\sum_{\substack{n \leq X\\ \left(n, P\right)=1}} f(n) \lambda_{\phi}(n+h)
\end{equation}
and a combination of bilinear sums of type
\begin{equation}\label{eq-bilinear}
\sum_{m\leq X/N} \,\sum_{N<p\leq N+M} f(m)f(p)\lambda_{\phi}(pm+h),
\end{equation}
where $(\log X)^A\ll N\ll X^{\delta}$ and $M\asymp \sqrt{N}$.  After applying the Cauchy--Schwarz inequality, the bilinear sum \eqref{eq-bilinear} is generally reduced to estimate
\[
\sum_{\substack{N<p_1,p_2\leq N+M \\ p_1\neq p_2}} f(p_1)f(p_2) \sum_{m\leq X/N} \lambda_{\phi}(p_1m+h)\lambda_{\phi}(p_2m+h).
\]
To obtain cancellations
in the innermost sum, we use the work \cite{Harcos2003} of Harcos on the shifted
convolution problem. The sifted sum \eqref{eq-seive} is a standard sieve condition, and is estimated by Hypothesis (ii) and Lemma \ref{lem-sieve-rs}. To prove Lemma \ref{lem-sieve-rs}, we need to investigate the asymptotic behavior of sum
\[
\sum_{\substack{n\leq X\\  n\equiv a \moq}}\lambda_{\phi}(dn)^2,
\]
where $d$ is a square-free integer and $(ad,q)=1$.

\section{Applications}
In this section, we will  provide some examples that fit into the framework of Theorem \ref{thm-main}.

\vskip 2mm
\subsection{Coefficients of ${\rm GL}_m$ automorphic $L$-functions} Let $m\geq 2$, and let $\mathcal{A}(m)$ be the set of all cuspidal automorphic representations of $\mathrm{GL}_m$ over $\QQ$ with unitary central character. For each $\pi\in \mathcal{A}(m)$, the corresponding $L$-function is defined by absolutely convergent Dirichlet series as
\begin{equation*}
L(s,\pi) =\sum_{n=1}^{\infty}\frac{\lambda_{\pi}(n)}{n^s}
\end{equation*}
for $\Re s >1$. We refer the reader to \cite[Section 3]{JLW} for a more detailed overview. Here we are mainly concerned with the problem of estimating correlations.
\begin{equation}\label{eq-SCP-m2}
\sum_{n \leq X}\lambda_{\pi}(n)\lambda_{\phi}(n+h),
\end{equation}
where $\phi$ is an arbitrary Hecke holomorphic or Hecke--Maass cusp form on ${\rm SL}_2(\ZZ)$. The sum \eqref{eq-SCP-m2} is often called the shifted convolution sum for ${\rm GL}_m\times {\rm GL}_2$. The shifted convolution problem is to strengthen the trivial bound, which means to prove
\[
\sum_{n \leq X}\lambda_{\pi}(n)\lambda_{\phi}(n+h)=o\Big(\sum_{n \leq X}|\lambda_{\pi}(n)\lambda_{\phi}(n+h)|\Big).
\]

When $\pi\in \mathcal{A}(2)$, this sum \eqref{eq-SCP-m2} has been investigated extensively by many authors since Selberg's seminal paper. Non-trivial bounds of this sum often have deep implications, for example, moment bounds and subconvexity for $L$-functions, equidistribution of Heegner points and quantum unique ergodicity.  In particular, consider the case $\pi$ induced from $\phi$. The best-known result in the literature is
\[
\sum_{n \leq X}\lambda_{\phi}(n)\lambda_{\phi}(n+h)\ll_{\phi,\varepsilon} X^{\frac{2}{3}+\varepsilon},
\]
where $0<|h|\leq X^{2/3}$. One can refer to Jutila \cite{Jutila1996} for details. If $\pi$ comes from a Hecke--Maass cusp form on ${\rm SL}_3(\ZZ)$,  Munshi \cite{Munshi2013} explored a smooth version of  \eqref{eq-SCP-m2} and established a power saving estimate over the trivial bound. These results for shifted convolution sums are usually treated using either the circle method or the spectral method. The basic strategy of circle method is to use the Kronecker delta symbol in order to separate the variable $h$ from $n$. There are different expressions for Kronecker delta symbol, such as $\delta$ method of Duke--Friedlander--Iwaniec, Jutila'a variant of $\delta$ method, and Heath's version of the Kloosterman circle method. While the spectral method is to investigate the analytic properties or spectral decomposition of corresponding Dirichlet series for the shifted convolution sum. To the best of our knowledge, however, for any $\pi\in \mathcal{A}(m)$ with $m\geq 4$, there is no non-trivial estimates for \eqref{eq-SCP-m2}  in the literature so far.

There are some stronger hypotheses in the beautiful work of Iwaniec, Luo and Sarnak \cite{IwLuoSa-2000}. One of them is Hypothesis C, which states that for $0<h_{1} \leqslant X$, and $0<\left|h_{2}\right| \leqslant X$, one has
$$
\sum_{m \leqslant \mathrm{X}} \mu(m) \lambda_{\phi}(m) \lambda_{\phi}\left(h_{1} m+h_{2}\right) \ll \mathrm{X}^{1 / 2+\varepsilon} .
$$
They remarked that Hypothesis C with the exponent $1/2$ replaced by $1-\delta$ for $\delta > 0$ is only needed, that is to say, any small power saving is enough for their application. We hope to take at least one small step forward on this hypothesis, even though we obtain a small power saving of $\log X$. Motivated by this discussion, another sum of interest is
\begin{equation}\label{eq-sum-mu-m2}
\sum_{n \leq X}\mu(n)\lambda_{\pi}(n)\lambda_{\phi}(n+h),
\end{equation}
where $\pi\in \mathcal{A}(m)$ with $m\geq 2$. It is worth mentioning a  surprising result of Pitt \cite{Pitt-2013}, who showed that
\begin{equation}\label{eq-pitt}
\sum_{n \leq X}\mu(n)\lambda_{\phi}(n+h)\ll X^{1-\delta}
\end{equation}
unconditional holds for $h=-1$, where $\delta$ is some positive constant, and $\phi$ is a Hecke holomorphic cusp form on ${\rm SL}_2(\ZZ)$. In fact, Pitt's result may be viewed as the simplest special case of sum \eqref{eq-sum-mu-m2} for $\pi\in \mathcal{A}(1)$.

In order to apply our theorem \ref{thm-main} to the sums \eqref{eq-SCP-m2} and \eqref{eq-sum-mu-m2}, we need to verify these two hypotheses. By the Rankin--Selberg theory and the inequality $|\lambda_{\pi}(n)|^2\leq \lambda_{\pi\times \tilde{\pi}}(n)$ for all positive integers $n$ (see \cite[Lemma 3.1]{JLW}), one has
\[
\sum_{n\leq X}|\lambda_{\pi}(n)|^2\leq \sum_{n\leq X} \lambda_{\pi\times \tilde{\pi}}(n)\ll_{\pi} X.
\]
Moreover, by employing the sieve technique, we have
\[
\sum_{\substack{n\leq X \\ (n, P(Y,Z))=1}}|\lambda_{\pi}(n)|^2\ll_{\pi} X\frac{\log Y}{\log Z}
\]
for any $N_{\pi}<Y<Z \leq X^{1/(30m^2)}$, which has been shown in \cite[Lemma 5.3]{JL-2021}. Here $N_{\pi}$ is the arithmetic conductor of $\pi$ and $P(Y,Z)$ is the product of these primes $p$ which belong to the interval $[Y,Z]$. For any fixed $\pi\in \mathcal{A}(m)$, Hypotheses (i) and (ii) then hold with $c=0, \gamma=1, Y= (\log X)^{32}$ and $Z= X^{1/(30m^2)}$. Therefore, Theorem \ref{thm-main} applies the multiplicative functions $\lambda_{\pi}(n)$ and $\mu(n)\lambda_{\pi}(n)$, and then yields the following result.

\begin{thm}\label{thm-main-2}
	Let $\lambda_{\phi}(n)$ be the Fourier coefficients of a  Hecke holomorphic or Hecke--Maass cusp form on ${\rm SL}_2(\ZZ)$, and $\lambda_{\pi}(n)$ be the coefficients of $L(s,\pi)$ for any fixed $\pi\in \mathcal{A}(m)$ with $m\geq 2$. For any non-zero $h\ll X$, we have
	$$
	\sum_{n \leq X}\lambda_{\pi}(n)\lambda_{\phi}(n+h)\ll X\frac{\log\log X}{\log X}
	$$
	and
	$$
	\sum_{n \leq X}\mu(n)\lambda_{\pi}(n)\lambda_{\phi}(n+h)\ll X\frac{\log\log X}{\log X},
	$$
	where the implied constant depends on $\phi$ and $\pi$ only.
\end{thm}
\begin{rem}
	If the multiplicative functions $\lambda_{\pi}(n)$ and $\mu(n)\lambda_{\pi}(n)$ are replaced by their absolute values $|\lambda_{\pi}(n)|$ and $|\mu(n)\lambda_{\pi}(n)|$, the corresponding estimates still hold with the same qualities. Especially, we get
	\[
	\sum_{n \leq X}|\lambda_{\pi}(n)| \, \lambda_{\phi}(n+h)\ll X\frac{\log\log X}{\log X}
	\]
	uniformly holds in $0<|h|\ll X$. When $\pi\in \mathcal{A}(2)$ comes from the cusp form $\phi$, Holowinsky  \cite{Holowi2009} investigated the behavior of the sum among two terms by taking absolute values and obtained
	\[
	\sum_{n \leq X}|\lambda_{\phi}(n)\, \lambda_{\phi}(n+h)| \ll \frac{X}{(\log X)^\delta},
	\]
	where $\delta$ is some positive constant. He also emphasized that the method presented in his article can not obtain a full integral power saving of $\log X$, that is, $\delta= 1$. In fact, this is determined by the Sate--Tate conjecture, which produces the best value $\delta= 2(1-8/(3\pi))\sim 0.30235$. Compared Holowinsky's result with that of ours, we obtain a full integral power saving of $\log X$ up to the small term $\log\log X$ for any $\pi\in \mathcal{A}(m)$ with $m\geq 2$. Of course, we also have to mention that the absolute value of at least one term in our summation is removed.
\end{rem}

\vskip 2mm

In order to illustrate that Theorem \ref{thm-main-2} provides a non-trivial saving, we need to estimate the correct order of magnitude of
\[
\sum_{n \leq X}|\lambda_{\pi}(n)\lambda_{\phi}(n+h)|.
\]
For simplicity of presentation,  we fix the shifted parameter $h=1$ for convenience and restrict the summation with $n$ and $n+h$ being square-free. Next, we use the standard probabilistic heuristics to predict the main term of sum
\[
S:=\sum_{n \leq X}|\mu(n)\lambda_{\pi}(n)\mu(n+1)\lambda_{\phi}(n+1)|.
\]
We refer to \cite[Section 4]{NgMark-2019} for a detailed explanation of the heuristics. Define two sequences of random variables $(X_{p})$ and $(Y_{p})$ by
$$
X_{p}(n)=\big|\mu(p^{\operatorname{ord}_{p}(n)})\,\lambda_{\pi}(p^{\operatorname{ord}_{p}(n)})\big|,  \quad  Y_{p}(n)=\big|\mu(p^{\operatorname{ord}_{p}(n+1)})\,\lambda_{\phi}(p^{\operatorname{ord}_{p}(n+1)})\big|,
$$
where $\operatorname{ord}_{p}(\cdot)$ is the $p$-adic valuation.
Associated to a random variable $Y: \mathbb{N} \rightarrow \mathbb{C}$ with image $\operatorname{im}(Y)=\{Y(n) \mid n \in \mathbb{N}\}$, its expected value to be
$$
\mathbb{E}(Y)=\sum_{a \in \operatorname{im}(Y)} a \cdot \mathbb{P}(Y=a),
$$
where for $B \subset\mathbb{N}$, the $\mathbb{P}(B)$ is given by
$$
\mathbb{P}(B)=\lim _{X \rightarrow \infty} \frac{\#\{1 \leq n \leq X \mid n \in B\}}{X} .
$$
Thus, we compute the expectations of $(X_{p}), (Y_{p})$ and $(X_{p}Y_{p})$ and then get
\begin{equation*}
\begin{aligned}
\mathbb{E}(X_{p}) =1+\Big(\frac{1}{p}-\frac{1}{p^2}\Big)(|\lambda_{\pi}(p)|-1),& \quad
\mathbb{E}(Y_{p})=1+\Big(\frac{1}{p}-\frac{1}{p^2}\Big)(|\lambda_{\phi}(p)|-1),\\
\mathbb{E}(X_{p}Y_{p})=1+\Big(\frac{1}{p}-\frac{1}{p^2}\Big)& (|\lambda_{\pi}(p)|+|\lambda_{\phi}(p)|-2).
\end{aligned}
\end{equation*}
With these notation as above, it is reasonable to make the following conjecture:
\begin{equation}\label{eq-S-conj}
S \sim \mathfrak{S}_{\pi, \phi}\Big(\sum_{n \leq X}|\mu(n)\lambda_{\pi}(n)|\Big)\,\Big(\sum_{n \leq X}|\mu(n)\lambda_{\phi}(n)|\Big) \frac{1}{X}
\end{equation}
as $X\rightarrow \infty$, where the singular series $\mathfrak{S}_{\pi, \phi}$ is equal to
\[
\mathfrak{S}_{\pi, \phi}=\prod_{p} \frac{\mathbb{E}(X_{p} Y_{p})}{\mathbb{E}(X_{p}) \mathbb{E}(Y_{p})}.
\]

  Assume that the Ramanujan conjecture holds for $\pi$ and $\phi$, that is
  $
  |\lambda_{\pi}(n)|\leq \tau_m(n)
  $
  and
  $
  |\lambda_{\phi}(n)|\leq \tau(n).
  $
  It is obvious that the singular series $\mathfrak{S}_{\pi, \phi}$ converges. We further suppose that a cuspidal automorphic representation $\pi$ comes from certain symmetric power lift of a holomorphic cusp form $\psi$ on ${\rm SL}_2(\ZZ)$.
  Let $\lambda_{{\rm sym}^{r} \psi}(n)$ denote the Dirichlet coefficients of its $r$-th symmetric power $L$-function $L_{r}(s, \psi)$, where $r\geq 1.$ Thanks to the Sato--Tate conjecture (which is now a theorem of Barnet-Lamb, Geraghty, Harris and Taylor \cite{BGHT}) and Wirsing's mean value theorem on multiplicative functions \cite[Theorem 1.1]{IK}, we have
  \begin{equation}\label{eq-sato-asmp}
  \sum_{n\leq X}|\mu(n)\lambda_{{\rm sym}^{r} \psi}(n)| \sim c_r(\psi)\frac{X}{\log^{\delta_r} X},
  \end{equation}
  where $\delta_r$ can be explicitly determined by
  \[
  \delta_r=1-\frac{2}{\pi}\int_{0}^{\pi} \frac{|\sin ((r+1) \theta)|}{\sin \theta} (\sin \theta)^2 \,\d \theta.
  \]
  A straightforward calculation of Maple gives
  \[
  \delta_r=1-\frac{4(r+1)}{r(r+2) \pi} \cot \Big(\frac{\pi}{2(r+1)}\Big).
  \]
  It is clear that $\delta_r$ is strictly increasing. Thus, for any $r \geq 1$, we have
  $$
  0.15<1-\frac{8}{3 \pi}=\delta_{1} \leq \delta_r\leq \lim _{r \rightarrow \infty} \delta_r=1-\frac{8}{\pi^{2}}<0.19.
  $$
  When this representation $\pi\in \mathcal{A}(m)$  corresponds to ${\rm sym}^{m-1} \psi$, one has $\lambda_{\pi}(n)=\lambda_{{\rm sym}^{m-1} \psi}(n)$. So we obtain from \eqref{eq-S-conj} and \eqref{eq-sato-asmp}  that
  \begin{equation*}
  S\sim  \big(\mathfrak{S}_{\pi, \phi} c_1(\psi)c_{m-1}(\psi)\big)\, \frac{X}{(\log X)^{\delta_{1}+\delta_{m-1}}}.
  \end{equation*}
  Finally, the probabilistic model gives
  \begin{equation*}
  \sum_{n \leq X}|\lambda_{\pi}(n)\lambda_{\phi}(n+1)|\gg \frac{X}{(\log X)^{\delta_{1}+\delta_{m-1}}}\gg \frac{X}{(\log X)^{3/5}},
  \end{equation*}
  where the implied constant depends on $\phi, m$. This shows that there exists a $(\log X)^{3/5}$ saving at least in Theorem \ref{thm-main} due to the cancellations of $\lambda_{\pi}(n)$ and $\lambda_{\phi}(n+1)$.

\vskip 5mm
\subsection{Divisor-bounded multiplicative functions}
Let $f(n)$ be any multiplicative function which satisfies $|f(n)|\leq \tau_m(n)$ for some $m>0$. Here $\tau_m(n)$ is the generalized divisor function, which equals to the number of ways of writing $n$ as the product of $m$ positive integers. If the function $\tau_m(n)$ meets Hypotheses (i) and (ii), then Theorem \ref{thm-main} holds for any multiplicative function $f(n)$ with $|f(n)|\leq \tau_m(n)$. First, one has the elementary bound
\[
\sum_{n\leq X}\tau_m(n)^2 \ll X(\log X)^{m^2-1},
\]
where the implied constant depends on $m$. Then it suffices to check that $\tau_m(n)$ satisfies Hypothesis (ii). One may use sieve method to get this condition similar to the proof of \cite[Lemma 5.3]{JL-2021}. However, we here provide an alternate proof, which can be directly derived from the  Brun--Titchmarsh theorem on multiplicative functions. By \cite[Theorem 1]{Shiu}, we have
\begin{equation}\label{divisor-BT}
	\sum_{\substack{n\leq X \\ (n, P(Y,Z))=1}}\tau_m(n)^2\ll \frac{X}{\log X} \exp\Big(m^2 \sum_{\substack{p\leq X\\ p\nmid P(Y,Z)}}\frac{1}{p}\Big),
\end{equation}
where $2\leq Y<Z\leq X$. Mertens' theorem then yields
\begin{equation*}\label{sel-merten-6}
	\begin{aligned}
		\sum_{\substack{p\leq X \\ p\nmid P(Y,Z)}}\frac{1}{p}=&\sum_{p\leq X}\frac{1}{p}-\sum_{p\leq Z}\frac{1}{p}+
		\sum_{p\leq Y}\frac{1}{p}\\
		=&\log\log X-\log\log Z+\log\log Y+O(1).
	\end{aligned}
\end{equation*}
Inserting this into \eqref{divisor-BT}, it follows
\[
\sum_{\substack{n\leq X \\ (n, P(Y,Z))=1}}\tau_m(n)^2\ll X(\log X)^{m^2-1}\Big(\frac{\log Y}{\log Z}\Big)^{m^2}
\]
for any $2\leq Y<Z\leq X$. Hypotheses (i) and (ii) then hold with $c=m^2-1, \gamma=m^2,  Y= (\log X)^{16(m^2+1)}$ and $Z= X^{1/23}$. Therefore Theorem \ref{thm-main} applies to any multiplicative function $f(n)$ with $|f(n)|\leq \tau_m(n)$ for some $m\geq 0$, and then yields the following result.

\begin{thm}\label{thm-main-3}
	Let $\lambda_{\phi}(n)$ be the Fourier coefficients of a  Hecke holomorphic or Hecke--Maass cusp form on ${\rm SL}_2(\ZZ)$, and $f(n)$ be any multiplicative function which satisfies $|f(n)|\leq \tau_m(n)$ for some $m\geq 0$. For any non-zero $h\ll X$, we have
	$$
	\sum_{n \leq X}f(n)\lambda_{\phi}(n+h)\ll X\frac{(\log\log X)^{\frac{m^2+1}{2}}}{\log X},
	$$
	where the implied constant depends on $\phi$ and $m$ only.
\end{thm}
\begin{rem}
	This theorem can be applied very directly to some functions of particular arithmetic interest:
	\begin{enumerate}
		\item the Piltz divisor function $\tau_z(n)$, where $\tau_z(n)$ is the Dirichlet coefficient of $\zeta(s)^z$ with $z\in\CC$;
		\item the number $a_n$ of integral ideals with norm $n$ in an algebraic number field $K$;
		\item the integral powers $\lambda_{\psi}(n)^l$ of Fourier coefficient of any holomorphic cusp form on $\Gamma_0(N)$, where $l\in \NN$.
	\end{enumerate}
\end{rem}

\begin{rem}
    If the Fourier coefficient $\lambda_{\phi}(n)$ is replaced by any multiplicative function $g$  with $|g(n)|\leq \tau_m(n)$, the corresponding sum has been studied
    by many number theorist such as Nair--Tenenbaum \cite{NairTen-1998}, Holowinsky \cite{Holo2010}, and Henriot \cite{Henriot-2012}. It admits a very accurately bound
    \[
    \sum_{n \leq X}f(n)g(n+h)\ll \Delta(|h|) \frac{X}{(\log X)^{2}} \prod_{p \leq X}\left(1+\frac{\left|f(p)\right|}{p}\right)\left(1+\frac{\left|g(p)\right|}{p}\right),
    \]
    where the function $\Delta(|h|)$ is explicitly given in \cite{Henriot-2012} such that $\sum_{|h|\leq X}\Delta(|h|)\ll 1$. We note that these three works also employ some sieve techniques.

\end{rem}

\vskip 5mm

\section{Review of ${\rm GL}(2)$ automorphic forms}

In this section, we collect some facts and conventions about automorphic forms on ${\rm GL}(2)$, and present some estimates and formulae for their Fourier coefficients which will be used later.

Let $k$ be an even integer, and Let $\mathcal{S}_k$ denote the set of arithmetically normalized primitive cusp forms of weight $k$ on ${\rm SL}_2(\ZZ)$, which are eigenfunctions of all the Hecke operators. Any $\phi \in \mathcal{S}_k$ has a Fourier expansion at infinity given by
	\begin{equation*}\label{fun:holo}
	\phi(z)=\sum_{n=1}^\infty \lambda_{\phi}(n)n^\frac{k-1}{2}e(nz).
	\end{equation*}
	Let $\mathcal{B}_{r}$ be the space of (weight zero) real-analytic Maass forms of Laplacian eigenvalue $1/4+\mu^2$ on ${\rm SL}_2(\ZZ)$, which are also eigenfunctions of all the Hecke operators. Any $\phi \in \mathcal{B}_{r}$ has a Fourier expansion at infinity of the form
	\begin{equation*}\label{fun:maass}
	\phi(z)=\sqrt{y}\sum_{n\neq 0} \lambda_{\phi}(n)K_{ir}(2\pi\abs{n}y)e(nx),
	\end{equation*}
	where $K_{ir}$ is the $K\mbox{-}$Bessel function. The arithmetic normalization above means that $\lambda_{\phi}(1)=1$. The eigenvalues $\lambda_{\phi}(n)$ enjoy the multiplicative property
	\begin{equation*}\label{Fc-mul-pro}
	\lambda_{\phi}(m)\lambda_{\phi}(n)=\sum_{d|(m,n)}\lambda_{\phi}\Big(\frac{mn}{d^2}\Big)
	\end{equation*}
	for all integers $m, n\geq 1.$ It follows from Rankin--Selberg theory that the Fourier coefficients $\lambda_{\phi}(n)$ are bounded on average, namely
	\begin{equation}\label{rs-gl2}
		\sum_{n\leq X}|\lambda_{\phi}(n)|^2=c_f\,X+O\big(X^{\frac{3}{5}}\big)
	\end{equation}
	for some $c_f>0$. For individual terms, we have
		\begin{equation}\label{kimsarnak-bound}
		|\lambda_{\phi}(n)|\leq n^{\frac{7}{64}}\tau(n)
	\end{equation}
	for all $n\geq1$ by the work \cite{ks} of Kim and Sarnak, where $\tau(n)$ is the divisor function. Moreover, if $\phi \in \mathcal{S}_k$, it follows from Deligne's proof \cite{D1974} of the Ramanujan--Petersson conjecture that
	\begin{equation*}\label{d-bound}
	|\lambda_{\phi}(n)|\leq \tau(n)
	\end{equation*}
hold for all $n\geq1$. It is also well known that the Fourier coefficients oscillate quite substantially. For example, the classical estimate of Wilton's type gives
	\begin{equation*}\label{wilton-estimate}
		S_{\phi}(\alpha, X):=\sum_{n \leq X} \lambda_{\phi}(n) e(n \alpha) \ll X^{\frac{1}{2}}\log X
    \end{equation*}
	uniformly in $\alpha\in \RR$, where the implied constant depends only on the form $\phi$. The explicit dependence on $\phi$ is  also studied by some number theorists. We note that this estimate plays an important role in solving shifted convolution problems, especially when the Jutila circle method is used.

For any  $\phi \in \mathcal{S}_k \cup \mathcal{B}_{r}$, the Hecke $L$-function is defined by
\[
L(s,\phi)=\sum_{n=1}^\infty \frac{\lambda_{\phi}(n)}{n^s}.
\]
This has an Euler product $L(s,\phi)= \prod\limits_{p}L_p(s,\phi)$ with the local factors
\[
L_p(s,\phi)=\big(1-\lambda_{\phi}(p)p^{-s}+p^{-2s}\big)^{-1}=\big(1-\alpha_{\phi}(p)p^{-s}\big)^{-1}\big(1-\beta_{\phi}(p)p^{-s}\big)^{-1},
\]
where $\alpha_{\phi}(p),\,\beta_{\phi}(p)$ are complex numbers with $\alpha_{\phi}(p)+\beta_{\phi}(p)=\lambda_{\phi}(p)$ and $\alpha_{\phi}(p)\beta_{\phi}(p)=1$.
%For all $p$ and $m\geq 0$, we have
%\begin{equation}\label{lambda-ab}
%	\lambda_{\phi}(p^m)=\sum_{0\leq j\leq m}\alpha_{\phi}(p)^{m-j}\beta_{\phi}(p)^{j}.
%\end{equation}
%The Rankin--Selberg $L$-function $L(s, f\otimes  f)$ is given by
%\begin{equation}
%	L(s, {\phi}\otimes {\phi})=\sum_{n=1}^\infty \frac{\lambda_{\phi}(n)^2}{n^s}.
%\end{equation}
%This has the Euler product $L(s, f\otimes  f):=\prod\limits_{p}L_p(s, {\phi}\otimes {\phi})$ with the local factors
%\[
%L_p(s, {\phi}\otimes {\phi})=\big(1-p^{-2s}\big)\big(1-\alpha_{\phi}(p)^2p^{-s}\big)^{-1}\big(1-p^{-s}\big)^{-2}\big(1-\beta_{\phi}(p)^2p^{-s}\big)^{-1}.
%\]
%It follows from \eqref{lambda-ab} that the $L_p(s, {\phi}\otimes {\phi})$ can be rewritten as
%\begin{equation}\label{local-ff2}
%	L_p(s, {\phi}\otimes {\phi})=\big(1-\lambda_{\phi}(p)^2p^{-s}+\lambda_{\phi}(p^2)p^{-2s}+\lambda_{\phi}(p)^2p^{-3s}-p^{-4s}\big)^{-1}.
%\end{equation}
%Define the symmetric square $L$-function by means of the Euler product $L(s, \text{sym}^2 \phi )=L_p(s, \text{sym}^2 \phi)$ with the local factors given by
%\[
%	L_p(s, \text{sym}^2 \phi)=
%\left(1-\alpha_{\phi}(p)^{2} p^{-s}\right)^{-1}\left(1-p^{-s}\right)^{-1}\left(1-\beta_{p}^{2} p^{-s}\right)^{-1}.
%\]
% This can be also written in terms of the Hecke eigenvalues
%\[
%	L_p(s, \text{sym}^2 \phi)=
%	\left(1-\lambda_{\phi}\left(p^{2}\right) p^{-s}+ \lambda_{\phi}\left(p^{2}\right) p^{-2 s}-p^{-3 s}\right)^{-1}.
%\]

Let $\chi$ be a primitive Dirichlet character modulo $q$ with $q\geq 1$. Then we define the twisted symmetric square $L$-function by means of the Euler product $L(s, \text{sym}^2 \phi\otimes \chi )=L_p(s, \text{sym}^2 \phi\otimes \chi)$ with the local factors given by
\[
L_p(s, \text{sym}^2 \phi\otimes \chi)=
\left(1-\alpha_{\phi}(p)^{2}\chi(p) p^{-s}\right)^{-1}\left(1-\chi(p)p^{-s}\right)^{-1}\left(1-\chi(p)\beta_{p}^{2} p^{-s}\right)^{-1}.
\]
It is well-known that this $L$-function extends to an entire function and satisfies a functional equation. Indeed, we have a completed $L$-function defined as
\[
\Lambda\left(s, \text{sym}^2 \phi\otimes \chi\right)=q^{3 s / 2} \gamma(s,\text{sym}^2 \phi\otimes \chi) L\left(s, \operatorname{Sym}^{2} f \otimes \chi\right),
\]
where $\gamma(s,\text{sym}^2 \phi\otimes \chi)$ is essentially a product of three gamma functions $\Gamma\left(\frac{s+\kappa_{j}}{2}\right), j=1,2,3$, with $\kappa_{j}$ depending on the weight or spectral parameter of $\phi$ and the parity of the character $\chi$, such that
$$
\Lambda\left(s, \text{sym}^2 \phi\otimes \chi\right)=\varepsilon(\text{sym}^2 \phi\otimes \chi) \Lambda\left(1-s, \text{sym}^2 \phi\otimes \bar{\chi}\right)
$$
Here $\Re \left(\kappa_{j}\right)>0$ and the $\varepsilon$-factor satisfies $|\varepsilon(\text{sym}^2 \phi\otimes \chi) |=1$.  Consider the Dirichlet series
\begin{equation*}\label{rstwist}
	L(s, {\phi}_\chi\otimes {\phi})=\sum_{n=1}^\infty \frac{\lambda_{\phi}(n)^2\chi(n)}{n^s}
\end{equation*}
for $\Re s>1.$
This has the Euler product $L(s, {\phi}_\chi\otimes {\phi}):=\prod\limits_{p}L_p(s, {\phi}_\chi\otimes {\phi})$ with the local factors
\[
L_p(s, {\phi}_\chi\otimes {\phi})=\big(1+\chi(p)p^{-s}\big)\big(1-\alpha_{\phi}(p)^2\chi(p)p^{-s}\big)^{-1}\big(1-\chi(p)p^{-s}\big)^{-1}\big(1-\beta_{\phi}(p)^2\chi(p)p^{-s}\big)^{-1}.
\]
This can be also written in terms of the Hecke eigenvalues
\begin{equation*}\label{ff-EulerHecke}
	L_p(s, {\phi}_\chi\otimes {\phi})=\big(1+\chi(p)p^{-s}\big)\big(1-\lambda_{\phi}(p^2)\chi(p)p^{-s}+\lambda_{\phi}(p^2)\chi(p)^2p^{-2s}-\chi(p)^3 p^{-3s}\big)^{-1}.
\end{equation*}
Comparing the Euler products of $L(s, \text{sym}^2 \phi\otimes \chi )$ and $L(s, {\phi}_\chi\otimes {\phi})$, one has
\begin{equation}\label{decom-twist-RS}
	L(s, {\phi}_\chi\otimes {\phi})=L(2s,\chi)^{-1}L(s,\chi)L(s, \text{sym}^2 \phi\otimes \chi ).
\end{equation}
%Rankin and Selberg succeeded in showing that this series satisfies the functional equation
%\begin{equation}\label{rs-twist-fe}
%	\Lambda(s, f_\chi\times f)=q^{2 s}\gamma(s,f_\chi\times f)L(s, f_\chi\times f)=\Big(\frac{\sqrt{q}}{\tau(\chi)}\Big)^4 \, \Lambda(1-s, f_{\bar{\chi}}\times f),
%\end{equation}
%where $\tau(\chi)$ is the Gauss sum, and the gamma factor $\gamma(s,f_\chi\times f)$ of $f_\chi\times f$ is given by
%$$
%\gamma(s,f_\chi\times f)=\pi^{-2 s} \Gamma\Big(\frac{s+\delta}{2}\Big) \Gamma\Big(\frac{s+k+\delta}{2}\Big) \Gamma\Big(\frac{s+1+\delta}{2}\Big) \Gamma\Big(\frac{s+k-1+\delta}{2}\Big)
%$$
%if $\phi \in \mathcal{S}_k$, and
%$$
%\begin{aligned}
%	\gamma(s,f_\chi\times f)=\pi^{-2 s} \Gamma\Big(\frac{s+2i\mu+\delta}{2}\Big)  \Gamma\Big(\frac{s+\delta}{2}\Big)^2 \Gamma\Big(\frac{s-2i\mu+\delta}{2}\Big)
%\end{aligned}
%$$
%if $\phi \in \mathcal{B}_{r}$.
This series $L(s, {\phi}_\chi\otimes {\phi})$ has a holomorphic continuation to all $s\in \CC$ except for possible simple pole at $s = 1$. The latter simple pole can only occur if $q=1$, in which case $\chi\equiv 1$. For  convenience, we denote $L(s, {\phi}_\chi\otimes {\phi})=L(s, {\phi}\otimes  {\phi})$ and $L(s, \text{sym}^2 \phi\otimes \chi )=L(s, \text{sym}^2 \phi)$ when $\chi\equiv 1$.
%
%	We will need the following general Voronoi-type summation formulae.
%	
%\begin{lem}
%	Let $d$ and $q$ be coprime integers, and let $g$ be a smooth, compactly supported function on $(0, \infty)$. If $\phi \in \mathcal{S}_k$, then
%$$
%\sum_{n=1}^{\infty} \lambda_{\phi}(m) e_{q}(d m) g(m)=\sum_{m=1}^{\infty} \lambda_{\phi}(m) e_{q}(-\bar{d} m) \hat{g}(m),
%$$
%where
%$$
%\hat{g}(y)=\frac{2 \pi i^{k}}{q} \int_{0}^{\infty} g(x) J_{k-1}\left(\frac{4 \pi \sqrt{x y}}{q}\right) d x.
%$$
%If $\phi \in \mathcal{B}_{r}$, then
%$$
% \sum_{m=1}^{\infty} \lambda_{\phi}(m) e_{q}(d m) g(m)=\sum_{\pm} \sum_{m=1}^{\infty} \lambda_{\phi}(\mp m) e_{q}(\pm \bar{d} m) g^{\pm}(m),
%$$
%where
%$$
%\begin{aligned}
%&g^{-}(y)=-\frac{\pi}{q \cosh \pi \mu} \int_{0}^{\infty} g(x)\left\{Y_{2 i \mu}+Y_{-2 i \mu}\right\}\left(\frac{4 \pi \sqrt{x y}}{q}\right) d x, \\
%&g^{+}(y)=\frac{4 \cosh \pi \mu}{q} \int_{0}^{\infty} g(x) K_{2 i \mu}\left(\frac{4 \pi \sqrt{x y}}{q}\right) d x .
%\end{aligned}
%$$
%Here $\bar{d}$ is a multiplicative inverse of $d \bmod q, e_{q}(x)=e(x / q)=e^{2 \pi i x / q}$ and $J_{k-1}, Y_{\pm 2 i \mu}, K_{2 i \mu}$ are Bessel functions.
%\end{lem}

\section{Some lemmas}

\subsection{Shifted convolution Problems}
In this section, we shall discuss the uniform bounds for shifted convolution sum
\[
\sum_{n=1}^{\infty} \lambda_{\phi}(m_1n+h)\lambda_{\phi}(m_2n+h) e^{-\frac{n}{X}}.
\]
If $m_1, m_2$ are distinct square-free integers, the case of $h=-1$ has been treated in the work \cite[Theorem 1.4]{Pitt-2013} of Pitt, which plays an important role to estimate the bilinear forms in the proof of upper bound \eqref{eq-pitt}. Recently, Assing, Blomer and Li \cite[Theorem 2.6]{ABL-2021} generalized Pitt's estimate to the case of general $h$ and gave the explicit dependences on the form $\phi$ and all shifts $m_1,m_2, h$. Note that the weight function $e^{-\frac{n}{X}}$ is replaced by a smooth function $g$ with support in $[X, 2X]$ and $g^{j}(x)\ll_{j} (X/Z)^{-j}$ for all $j\geq 0$ in \cite[Theorem 2.6]{ABL-2021}. Next, in view of Harcos's result \cite{Harcos2003}, we will prove another uniform estimate for the above sum, which is enough in our applications.

\begin{lem}\label{lem-shifted-sum}
	Let $\phi \in \mathcal{S}_k \cup \mathcal{B}_{r}$. Let $m_1,\,m_2$ be positive integers with $(m_1,m_2)=1$ and $h\neq 0$. Then we have
	\begin{equation*}
	\sum_{n=1}^{\infty} \lambda_{\phi}(m_1n+h)\lambda_{\phi}(m_2n+h) e^{-\frac{n}{X}}\ll e^{\frac{(m_1+m_2)h}{2Xm_1m_2}}(m_1m_2)^{\frac{4}{5}}X^{\frac{9}{10}+\varepsilon},
	\end{equation*}
	where the implied constant depends on $\varepsilon$ and the form $\phi$.
\end{lem}

\begin{proof}
	We denote by $\mathcal{C}$ the left side of the claimed inequality. Introducing the new variables $n_1=m_1n+h$ and $n_2=m_2n+h$, the sum $\mathcal{C}$ can be rewritten as
	\begin{equation}\label{C-shifted}
	\begin{aligned}
	\mathcal{C}
	=&e^{\frac{(m_1+m_2)h}{2Xm_1m_2}} 	\sum_{\substack{n_1,n_2=1\\ m_2(n_1-h)=m_1(n_2-h) }}^{\infty} \lambda_{\phi}(n_1)\lambda_{\phi}(n_2) e^{-\frac{n_1}{2Xm_1}}e^{-\frac{n_2}{2Xm_2}}\\
	:=&e^{\frac{(m_1+m_2)h}{2Xm_1m_2}}D_w\big(m_2,m_1;(m_2-m_1)h\big),
	\end{aligned}
	\end{equation}
	where the notation $D_w\big(m_2,m_1;(m_1-m_2)h\big)$ is defined as in the work \cite{Harcos2003} of Harcos
	\[
	D_w\big(m_2,m_1;(m_1-m_2)h\big)=	\sum_{\substack{n_1,n_2=1\\ m_2n_1-m_1n_2=(m_2-m_1)h}}^{\infty} \lambda_{\phi}(n_1)\lambda_{\phi}(n_2) w(m_2n_1,m_1 n_2)
	\]
	with the nice weight function
	\[
	w(x,y)=e^{-\frac{x+y}{2Xm_1m_2}}.
	\]
	Clearly, this weight function $w(x,y)$ satisfies the estimate
	\[
	x^i y^jw^{(i,j)}(x,y)\ll_{i,j}\Big(1+\frac{x}{Xm_1m_2}\Big)^{-1}\Big(1+\frac{y}{Xm_1m_2}\Big)^{-1}
	\]
	for all $i,j\geq 0$. By the nontrivial uniform estimate in \cite[Theorem 1]{Harcos2003}, we obtain that
	\[
	D_w\big(m_2,m_1;(m_2-m_1)h\big)\ll_{f,\varepsilon} (m_1m_2)^{\frac{4}{5}}X^{\frac{9}{10}+\varepsilon}
	\]
	holds for any $h\neq 0$. On inserting this estimate into \eqref{C-shifted}, this lemma follows.
\end{proof}

\vskip 5mm

%For the proof of Proposition \ref{lem-shifted-sum}, we apply Jutila's variant of the circle method.
%
%
%\subsection{Applying the circle method}
%Introducing the new variables $n_1=m_1n+h$ and $n_2=m_2n+h$, the sum above can be rewirten as
%\[
%	\sum_{n_1,n_2=1\atop m_2(n_1-h)=m_1(n_2-h) }^{\infty} \lambda_{\phi}(n_1)\lambda_{\phi}(n_2) V\Bigl(\frac{n_2-h}{m_2X} \Bigr)
%\]
%We will use Jutila's variant of the circle method. For any set $S \subset \mathbb{R}$, let $\mathbb{I}_{S}$ denote the associated characteristic function, that is, $\mathbb{I}_{S}(x)=1$ for $x \in S$ and 0 otherwise. For any collection of positive integers $\mathcal{Q} \subset[1, Q]$ (which we call the set of moduli) and a positive real number $\delta$ in the range $Q^{-2} \ll \delta \ll Q^{-1}$, we define the function
%\[
%	\tilde{I}
%\]
%
%
%\begin{lem}\label{jutila-circle}
%    Let $Q \geq 1$ and $Q^{-2} \leq \delta \leq Q^{-1}$ be two parameters. Let $w$ be a nonnegative function with support in $[Q, 2 Q]$ satisfying $\|w\|_{\infty} \leq 1$ and $\sum w(q)>0 .$ For each fraction $d/q$, write $I_{d/q}(\alpha)$ for the characteristic function of the interval $[d/q-\delta, d/q+\delta]$ and define
%$$
%L:=\sum_{q} w(q) \phi(q), \quad \tilde{I}(\alpha)=\frac{1}{2 \delta L} \sum_{q} w(q) \sideset{}{^*}\sum_{d \moq} \mathrm{I}_{d / q}(\alpha)
%$$
%Then $\tilde{\mathrm{I}}(\alpha)$ is a good approximation to the characteristic function on $[0,1]$ in the sense that
%$$
%\int_{0}^{1}(1-\tilde{I}(\alpha))^{2} d \alpha \ll_{\varepsilon} \frac{Q^{2+\varepsilon}}{\delta L^{2}}
%$$
%for any $\varepsilon>0$.
%\end{lem}

\subsection{Sieve conditions}

Let $\lambda_{\phi}(n)^2$ be the coefficients of Rankin--Selberg convolution $L$-function
$L(s,{\phi}\otimes {\phi})$. We denote by $P(Y,Z)$ the product of these primes $p$ which belong to the interval $[Y,Z]$,
where $2\leq Y<Z$. We emphasize that there is no restriction on $Y$ and $Z$ as in Hypothesis (ii).
In this section, we shall seek an upper estimate for the sifted sum
\begin{equation}\label{sifted-pi}
	\sum_{\substack{n \leq X\\ \left(n-a, P(Y,Z)\right)=1}}\lambda_{\phi}(n)^2
\end{equation}
with $0<|a|\ll X$.
For this purpose, we need to investigate the distribution for $\lambda_{\phi}(n)^2$ in arithmetic progressions. This means to give an asymptotic formula for
\[
\sum_{\substack{n\leq X\\  n\equiv a \moq}}\lambda_{\phi}(n)^2,
\]
where $q$ increases together with $X$. In our application to the proof of Theorem \ref{thm-main}, the modulo $q\ll X^\varepsilon$ is sufficient. However, we hope to get the modulo $q$ as large as possible such that the sum above admits an asymptotic formula, which may be of independent interest. Since there is no restriction $(a,q)=1$ in arithmetic progressions, we need to prove the following.

%\subsection{Second moment of Fourier coefficients in arithmetic progression}
\begin{prop}\label{prop-rs-ap}
	Let $\phi \in \mathcal{S}_k \cup \mathcal{B}_{r}$, $d$ be a square-free integer and $(ad,q)=1$. Then we have
	\[
	\sum_{\substack{n\leq X\\  n\equiv a \moq}}\lambda_{\phi}(dn)^2= \frac{1}{\varphi(q)} H_d(1, {\phi}\otimes {\phi})G_q(1, {\phi}\otimes {\phi})\mathop{\textup{Res}}_{s=1} L(s,{\phi}\otimes {\phi})X+O\Big(d^{\frac{7}{32}}X^{\frac{3}{5}+\varepsilon}\Big)
	\]
	for any $q<X^{\frac{2}{5}-\varepsilon}$,
	where
	$H_d(1, {\phi}\otimes {\phi})$ and $G_q(1, {\phi}\otimes {\phi})$ are given by
	\begin{equation*}
	\label{Gq}
	\begin{aligned}
	H_d(1, {\phi}\otimes {\phi}) =&
	\prod_{p|d}\big(1+p^{-1}\big)^{-1}\big(\lambda_{\phi}(p)^2-\lambda_{\phi}(p^2)p^{-1}+p^{-2}\big),\\
	G_q(1, {\phi}\otimes {\phi}) =&
	\prod_{p|q}\big(1+p^{-1}\big)^{-1}\big(1-\lambda_{\phi}(p^2)p^{-1}+\lambda_{\phi}(p^2)p^{-2}- p^{-3}\big),
	\end{aligned}
	\end{equation*}
	and the implied constant depends only on $\phi, \varepsilon$.
\end{prop}

\begin{rem}
	Proposition \ref{prop-rs-ap} generalized the classical result \eqref{rs-gl2} of Rankin--Selberg to the situation of arithmetic progressions when $d=1$. Moreover, the error terms of these two results have the same quality up to the arbitrary small $\varepsilon.$
\end{rem}

\begin{proof}
	It is more convenient to deal with the related smoothing sums, so we introduce two smooth compactly supported functions $w_{\ell}$ inspired by the example in \cite[Appendix A]{Iwaniec-2014}. Let the function $g(x)$ be defined by
	\[
	g(x)=\left\{\begin{array}{ll}
	\exp (-1 / x(1-x)) & \text { if } 0<x<1, \\
	0 & \text { otherwise.}
	\end{array}\right.
	\]
	Put
	$$
	G(y)=\frac{1}{C} \int_{-\infty}^{y}g(x)\d x,
	$$
	where $C=\int_{0}^{1} g(x) \mathrm{d} x$ is an absolute constant. Using $G(y)$, we construct the two functions
	\[
	w_1(u)=G\left(\frac{u+Y}{Y}\right)-G\left(\frac{u-X}{Y}\right)
	\]
	and
	\[
	w_2(u)=G\left(\frac{u}{Y}\right)-G\left(\frac{u-X+Y}{Y}\right),
	\]
	where $0< Y<X/2$ and $Y$ will be determined later.  Then $w_{\ell}(u)$ satisfy the following properties:
	\begin{itemize}
		\item
		$w_{1}(u)=1$ for $u\in[Y,X-Y], w_{-}(u)=0$ for $u\geq X$ and $u\leq 0$;
		\item
		$w_{2}(u)=1$ for $u\in[0,X], w_{+}(u)=0$ for $u\geq X+Y$ and $u\leq -Y$;
		\item
		$w_{\ell}^{(j)}(u)\ll_j Y^{-j}$ for all $j\geq 0$ and $\ell=1,2$;
		\item the Mellin transform of $w_{\ell}(x)$ is
		\begin{equation}\label{UB_Mw}
		\begin{aligned}
		\widehat{w_{\ell}}(s)
		&= \int_0^\infty w_{\ell}(u)u^{s-1} \d u
		\\
		& = \frac{1}{s\cdots (s+j-1)}\int_0^\infty w_{\ell}^{(j)}(u)u^{s+j-1} \d u
		\\
		& \ll_j \frac{Y}{X^{1-\sigma}} \Big(\frac{X}{|s|Y}\Big)^j
		\end{aligned}
		\end{equation}
		for any $j \geq 1$ and $\sigma=\Re s> 0$.
		
		\item
		trivially $\widehat{w_{\ell}}(s)\ll X^\sigma$ for any $\sigma=\Re s> 0$ and
		\begin{equation}\label{X+Y}
		\widehat{w_{\ell}}(1)=X+O(Y).
		\end{equation}
	\end{itemize}

	We shall detect the congruence $n \equiv a \moq$ via the orthogonality of Dirichlet characters.  Thus, we obtain the identity
	\begin{equation}\label{sum-smooth-unsmooth}
	\begin{aligned}
	S(w_{\ell};q,a;d):=&\sum_{\substack{n\\  n\equiv a \moq}}\lambda_{\phi}(dn)^2w_{\ell}(n)=\frac{1}{\varphi(q)}\sum_{\chi \moq }\overline{\chi}(a)\sum_{n}\lambda_{\phi}(dn)^2\chi(n)w_{\ell}(n)\\
	=&\frac{1}{\varphi(q)} M(w_{\ell};q;d)+O\big(E(w_{\ell};q;d)\big),
	\end{aligned}
	\end{equation}
	where $M(w_{\ell};q;d)$ and $E(w_{\ell};q;d)$ are given by
	\begin{equation}\label{eq-defME}
	\begin{aligned}
	M(w_{\ell};q;d)=&	\sum_{\substack{n\\ (n,q)=1}}\lambda_{\phi}(dn)^2w_{\ell}(n),\\
	E(w_{\ell};q;d)=&\frac{1}{\varphi(q)}	\sum_{\substack{\chi \moq \\ \chi \neq \chi_0 }}\Big|\sum_{n}\lambda_{\phi}(dn)^2\chi(n)w_{\ell}(n)\Big|.
	\end{aligned}
	\end{equation}
	
	Now it suffices to estimate $M(w_{\ell};q;d)$ and $E(w_{\ell};q;d)$, respectively.
	
	\begin{lem}\label{Main-term}
		Let $M(w_{\ell};q;d)$ be defined as above. Then we have
		\[
		\begin{aligned}
		M(w_{\ell};q;d)=&H_d(1, {\phi}\otimes {\phi}) G_q(1, {\phi}\otimes {\phi})\mathop{\textup{Res}}_{s=1} L(s,{\phi}\otimes {\phi})X+O\big(d^{\frac{7}{32}}\tau(dq)  Y^{-\frac{1}{4}} X^{\frac{3}{4}}(\log X)^3\big)\\
		&+O\big(d^{\frac{7}{32}}\tau(dq)Y\big),
		\end{aligned}
		\]
		where the implied constant depends only on $\phi$.
	\end{lem}

	\begin{proof}
		By the Mellin inversion formula
		\[
		w_{\ell}(x) = \frac{1}{2\pi i} \int_{2- i \infty}^{2+i \infty} \widehat{w_{\ell}}(s) x^{-s} \d s,
		\]
		we may write
		\[
			\sum_{\substack{n\\ (n,q)=1}}\lambda_{\phi}(dn)^2w_{\ell}(n) = \frac{1}{2\pi i } \int_{2-i\infty}^{2+i\infty} \widehat{w_{\ell}}(s) 	\sum_{\substack{n=1 \\ (n,q)=1}}^{\infty}
		\frac{\lambda_{\phi}(dn)^2}{n^s} \d s.
		\]
		The Dirichlet series appearing above has non-negative coefficients, converges in the region $\Re s >1$, and
		matches the Rankin--Selberg convolution $L$-function $L(s, {\phi}\otimes {\phi})$ except for the Euler factors at primes $p$ dividing $dq$.
		Note that $d$ is square-free and $(d,q)=1$. By the multiplicative property of $\lambda_{\phi}(n)$, we may rewrite
		\begin{equation*}
		\begin{aligned}
		\sum_{\substack{n=1 \\ (n,q)=1}}^\infty \frac{\lambda_{\phi}(dn)^2}{n^s} =&\prod_{p| d}\Big(\sum_{k=0}^{\infty} \lambda_{\phi}(p^{k+1})^2p^{-k s}\Big)\prod_{p\nmid  dq}L_p(s, {\phi}\otimes {\phi})\\
		=&L(s,{\phi}\otimes {\phi})\prod_{p| d}\Big(L_p(s, {\phi}\otimes {\phi})^{-1}\sum_{k=0}^{\infty} \lambda_{\phi}(p^{k+1})^2p^{-k s}\Big) \prod_{p|q} L_p(s, {\phi}\otimes {\phi})^{-1}\\
		=&L(s,{\phi}\otimes {\phi})H_d(s, {\phi}\otimes {\phi})G_q(s, {\phi}\otimes {\phi}),
		\end{aligned}
		\end{equation*}
		where $H_d(s, {\phi}\otimes {\phi})$ and $G_q(s, {\phi}\otimes {\phi})$ are given by
		\begin{equation*}
		\label{Gq}
		\begin{aligned}
		H_d(s, {\phi}\otimes {\phi}) =&
		\prod_{p|d}\big(1+p^{-s}\big)^{-1}\big(\lambda_{\phi}(p)^2-\lambda_{\phi}(p^2)p^{-s}+p^{-2s}\big),\\
		G_q(s, {\phi}\otimes {\phi}) =&
		\prod_{p|q}\big(1+p^{-s}\big)^{-1}\big(1-\lambda_{\phi}(p^2)p^{-s}+\lambda_{\phi}(p^2)p^{-2s}- p^{-3s}\big).
		\end{aligned}
		\end{equation*}
		Therefore, the above integral equals
		\begin{equation}\label{int-L}
		\frac{1}{2\pi i } \int_{2-i\infty}^{2+i\infty} \widehat{w_{\ell}}(s) L(s, {\phi}\otimes {\phi})H_d(s, {\phi}\otimes {\phi}) G_q(s, {\phi}\otimes {\phi}) {\rm{d}}s.
		\end{equation}

	We evaluate \eqref{int-L} by moving the line of integration to $\Re s=1/2$.  We encounter a simple pole at $s=1$ and the residue
	here is
	\[
	\widehat{w_{\ell}}(1)H_d(1, {\phi}\otimes {\phi})G_q(1, {\phi}\otimes {\phi})\mathop{\textup{Res}}_{s=1}  L(s, {\phi}\otimes {\phi}).
	\]
	It follows from \eqref{kimsarnak-bound} and the definitions of $H_d(s, {\phi}\otimes {\phi}), G_q(s, {\phi}\otimes {\phi})$ that
	\begin{equation}\label{gdq-1}
	\begin{aligned}
	H_d(s, {\phi}\otimes {\phi})\ll& \prod_{p|d} \big(p^{\frac{7}{32}}+p^{-\frac{9}{32}+\varepsilon}\big)\ll d^{\frac{7}{32}}\tau(d),\\
	G_q(s,{\phi}\otimes {\phi}) \ll& \prod_{p|q}\big(1+p^{-\frac{9}{32}+\varepsilon}\big) \ll \tau(q)
	\end{aligned}
	\end{equation}
	for any $\Re s\geq 1/2$.
	Combining these with \eqref{X+Y}, the residue equals to
	\begin{equation}\label{res-value}
	H_d(1, {\phi}\otimes {\phi})G_q(1, {\phi}\otimes {\phi})\mathop{\textup{Res}}_{s=1} L(s,{\phi}\otimes {\phi})X+O\big(d^{\frac{7}{32}}\tau(dq)Y\big).
	\end{equation}
	It remains to bound the integral on the line $\Re s=1/2$, that is
	\begin{equation}\label{int-12}
	\begin{aligned}
	&\frac{1}{2\pi i } \int_{\frac{1}{2}-i\infty}^{\frac{1}{2}+i\infty} \widehat{w_{\ell}}(s) L(s, {\phi}\otimes {\phi}) G_q(s, {\phi}\otimes {\phi}) {\rm{d}}s\\
	\ll& d^{\frac{7}{32}}\tau(dq) \int_{-\infty}^{\infty}\Big| \zeta\Big(\frac{1}{2}+it\Big)L\Big(\frac{1}{2}+it, \text{sym}^2 \phi \Big)  \widehat{w_{\ell}} \Big(\frac 12 +it \Big) \Big| \log(|t|+2) {\rm{d}}t,
	\end{aligned}
	\end{equation}
	where we use the decomposition \eqref{decom-twist-RS} and the bound $\zeta(1+it)^{-1}\ll \log(|t|+2)$ for any $t\in \RR$. Moreover, the contribution from the integration over $|t| \geq T = X^{1+\varepsilon}/Y$ is negligibly small (i.e. $O(X^{-A})$ for any $A> 0$) if we choose a sufficiently large $j$ in \eqref{UB_Mw}. Hence,
	the term in the second line of \eqref{int-12} is
	\begin{equation*}\label{int-122}
	\ll d^{\frac{7}{32}} \tau(dq)  (\log T)^2 X^{\frac{1}{2}}\max_{0\leq T_1\leq T}\frac{1}{T_1+1}\int_{T_1/2}^{T_1}\Big| \zeta\Big(\frac{1}{2}\pm it\Big)L\Big(\frac{1}{2}\pm it, \text{sym}^2 \phi \Big)  \Big|  {\rm{d}}t+X^{-A}.
	\end{equation*}
	Applying the Cauchy--Schwarz inequality and inserting the second moments of $\zeta(s)$ and $L(s, \text{sym}^2 \phi)$ with the $t$-aspect integration
	\[
	\int_{T/2}^{T}\Big| \zeta\Big(\frac{1}{2}\pm it\Big)\Big|^2  {\rm{d}}t\ll T\log T, \quad  \int_{T/2}^{T}\Big| L\Big(\frac{1}{2}\pm it, \text{sym}^2 \phi \Big)\Big|^2  {\rm{d}}t\ll T^{\frac{3}{2}}\log T,
	\]
	we obtain that the integral on the line $\Re s=1/2$ is less than
	$
	d^{\frac{7}{32}}\tau(dq)  T^{\frac{1}{4}}(\log T)^3 X^{\frac{1}{2}}.
	$
	Combining this with \eqref{res-value}, we get
	\[
	\begin{aligned}
	\sum_{\substack{n\\ (n,q)=1}}\lambda_{\phi}(n)^2w_{\ell}(n)=&H_d(1, {\phi}\otimes {\phi}) G_q(1, {\phi}\otimes {\phi})\mathop{\textup{Res}}_{s=1} L(s,{\phi}\otimes {\phi})X\\
	&+O\big(d^{\frac{7}{32}}\tau(dq)  Y^{-\frac{1}{4}} X^{\frac{3}{4}}(\log X)^3\big)+O\big(d^{\frac{7}{32}}\tau(dq)Y\big).
	\end{aligned}
	\]
	This completes the proof of Lemma \ref{Main-term}.
\end{proof}

\begin{lem}\label{Error-term}
Let $E(w_{\ell};q;d)$ be defined as in \eqref{eq-defME}. Then we have
\begin{equation*}
E(w_{\ell};q;d)\ll \Big(\frac{qX}{Y}\Big)^{\frac{1}{4}}d^{\frac{7}{32}} X^{\frac{1}{2}+\varepsilon},
\end{equation*}
where the implied constant depends only on $\phi, \varepsilon$.
\end{lem}
\begin{proof}
As before, we have
\[
\sum_{n}\lambda_{\phi}(dn)^2\chi(n)w_{\ell}(n)=\frac{1}{2\pi i } \int_{2-i\infty}^{2+i\infty} \widehat{w_{\ell}}(s)  \sum_{n=1}^\infty \frac{\lambda_{\phi}(dn)^2\chi(n)}{n^s}\d s.
\]
Each non-principle character $\chi\moq$ can be induced by a primitive character  $\chi \modq$ with $q_1|q$ and $q_1>1$. Then we get the Dirichlet series
\[
\sum_{n=1}^\infty \frac{\lambda_{\phi}(dn)^2\chi(n)}{n^s}=L(s, f_{\chi_1}\otimes f)H_q(s, f_{\chi_1} \otimes f)G_q(s, f_{\chi_1} \otimes f),
\]
where $H_q(s, f_{\chi_1} \otimes f)$ and $)G_q(s, f_{\chi_1} \otimes f)$ are given by
\begin{equation*}
\label{Gq}
\begin{aligned}
H_q(s, f_{\chi_1} \otimes f) =&
\prod_{p|d}\big(1+\chi_1(n)p^{-s}\big)^{-1}\big(\lambda_{\phi}(p)^2\chi_1(n)-\lambda_{\phi}(p^2)\chi_1(n)p^{-s}+\chi_1(n)p^{-2s}\big),\\
G_q(s, f_{\chi_1} \otimes f) =&
\prod_{p|q}\big(1+\chi_1(n)p^{-s}\big)^{-1}\big(1-\lambda_{\phi}(p^2)\chi_1(n)p^{-s}+\lambda_{\phi}(p^2)\chi_1(n)p^{-2s}- \chi_1(n)p^{-3s}\big).
\end{aligned}
\end{equation*}
Following the argument of Lemma \ref{Main-term}, one can derive
\[
\begin{aligned}
&\sum_{n}\lambda_{\phi}(dn)^2\chi(n)w_{\ell}(n)\\
\ll& d^{\frac{7}{32}} \tau(dq) X^{\frac{1}{2}+\varepsilon}\max_{0\leq T_1\leq T}\frac{1}{T_1+1}\int_{T_1/2}^{T_1}\Big| L\Big(\frac{1}{2}\pm it,\chi_1\Big)L\Big(\frac{1}{2}\pm it, \text{sym}^2 \phi\otimes {\chi_1} \Big)  \Big|  {\rm{d}}t,
\end{aligned}
\]
where $T=X^{1+\varepsilon}/Y.$ Then the Cauchy--Schwarz inequality gives
\[
E(w_{\ell};q;d)\ll \frac{ d^{\frac{7}{32}} \tau(dq) }{\varphi(q)} X^{\frac{1}{2}+\varepsilon}\max_{0\leq T_1\leq T}\frac{1}{T_1+1}\sum_{q_1|q} (I_1\,I_2)^\frac{1}{2}.
\]
Here $I_1, I_2$ are moments of $L(s,\chi_1)$ and $L(s, \text{sym}^2 \phi\otimes {\chi_1})$, respectively. More preciously,
\[
\begin{aligned}
I_1=&\sideset{}{^*}\sum_{\chi\modq }\int_{T_1/2}^{T_1}\Big| L\Big(\frac{1}{2}\pm it,\chi_1\Big) \Big|^2  {\rm{d}}t,\\
I_2=&\sideset{}{^*}\sum_{\chi\modq }\int_{T_1/2}^{T_1}\Big| L\Big(\frac{1}{2}\pm it,\text{sym}^2 \phi\otimes {\chi_1} \Big) \Big|^2  {\rm{d}}t.\\
\end{aligned}
\]
Recall a large sieve inequality \cite[Page 179]{IK}: for any complex number $a_n$, one has
\[
\sum_{\chi \moq}\Big|\sum_{n \leqslant N} a_{n} \chi(n)\Big|^{2} \leq(q+N)\sum_{n\leq N}|a_n|^2.
\]
Together this with the approximate functional equations of $L(s,\chi_1)$ and $L(s, \text{sym}^2 \phi\otimes {\chi_1})$, we have
\[
I_1\ll rT_1\log q_1T_1, \quad I_2\ll (q_1T_1)^{\frac{3}{2}}\log q_1T_1.
\]
Further, we obtain
\[
E(w_{\ell};q;d)\ll\Big(\frac{qX}{Y}\Big)^{\frac{1}{4}}d^{\frac{7}{32}} X^{\frac{1}{2}+\varepsilon}.
\]
\end{proof}

Now we continue to prove Proposition \ref{prop-rs-ap}. Inserting Lemma \ref{Main-term} and Lemma \ref{Error-term} into \eqref{sum-smooth-unsmooth} and taking $Y=qX^{\frac{3}{5}}$, we have
\[
S(w_{\ell};q,a;d)= \frac{1}{\varphi(q)} G_q(1, f \times f)\mathop{\textup{Res}}_{s=1} L(s,f \times f)X+O\Big(d^{\frac{7}{32}}X^{\frac{3}{5}+\varepsilon}\Big)
\]
for $\ell=1,2$ and any $q<X^{\frac{2}{5}-\varepsilon}$. By the definition of $w_{\ell}(u)$, it is easy to see
\[
S(w_{1};q,a;d)\leq \sum_{\substack{n\leq X\\  n\equiv a \moq}}\lambda_{\phi}(dn)^2\leq S(w_{2};q,a; d),
\]
which implies Proposition \ref{prop-rs-ap}.
\end{proof}

In order to estimate the sifted sum \eqref{sifted-pi} without the restriction $\big(a,P(Y,Z)\big)=1$, it is necessary to explore the distribution of $\lambda_{\phi}(n)^2$ in the general arithmetic progressions, which do not have the condition $(a,q)=1.$ Proposition \ref{prop-rs-ap} can yield the following result.

\begin{cor}
	\label{rs-ap-2}
	Let $\phi \in \mathcal{S}_k \cup \mathcal{B}_{r}$ and $q$ be a square-free integer. Then we have
	\[
	\sum_{\substack{n\leq X\\  n\equiv a \moq}}\lambda_{\phi}(n)^2= \frac{1}{q_0\varphi(q/q_0)} H_{q_0}(1, {\phi}\otimes {\phi})G_{q/q_0}(1, {\phi}\otimes {\phi})\mathop{\textup{Res}}_{s=1} L(s,{\phi}\otimes {\phi})X+O\Big(X^{\frac{3}{5}+\varepsilon}\Big)
	\]
	for any $q<X^{\frac{2}{5}-\varepsilon}$, where $q_0=(a,q)$, $H_{q_0}(1, {\phi}\otimes {\phi}),G_{q/q_0}(1, {\phi}\otimes {\phi})$ are as in Proposition \ref{prop-rs-ap}, and the implied constant depends only on $\phi,\varepsilon$.
\end{cor}
\begin{proof}
	Assume $a=a_0a_1$ and $q=q_0q_1$ with $(q_0a_1,q_1)=1$.
	Then we obtain the identical relation
	\[
	\sum_{\substack{n\leq X\\  n\equiv a \moq}}\lambda_{\phi}(n)^2=\sum_{\substack{n\leq X/q_0\\  n\equiv a_1 \modq}}\lambda_{\phi}(q_0n)^2,
	\]
	where $q_0=(a,q)$. This corollary immediately follows from Proposition \ref{prop-rs-ap}.
\end{proof}

Now we are ready to handle the sifted sum \eqref{sifted-pi} and get the following result.

\begin{lem}\label{lem-sieve-rs}
	Suppose that $0<Y<Z \leq X^{\frac{1}{25}-\varepsilon}$. Let $P(Y,Z)$ denote the product of primes $p$ which belong to the interval $[Y,Z)$. Then we have
	$$
	\sum_{\substack{n \leq X\\ \left(n-a, P(Y,Z)\right)=1}}\lambda_{\phi}(n)^2  \ll X  \frac{\log Y}{\log Z}
	$$
	uniformly in any integer $a$.
\end{lem}

\begin{proof}
	This lemma will follow from a standard application of sieve method and Proposition \ref{prop-rs-ap}. To estimate the contribution of main term in Proposition \ref{prop-rs-ap} to the sifted sum \eqref{sifted-pi}, we have to compute the sieve dimension. Additionally, the sieve dimension also directly determines the choice of sieve level. So we start to deal with a product of the type:
	\[
	\prod_{w\leq p<z}(1-g(p))^{-1},
	\]
	where $0<w<z$ and $g(p)$ is the density function. By Corollary \ref{rs-ap-2}, the density function at primes is given by
	\[
	g(p)=\begin{cases}
	\frac{1}{p+1}\Big(\lambda_{\phi}(p)^2-\frac{\lambda_{\phi}(p^2)}{p}+\frac{1}{p^2}\Big) \quad &\text{if } p\mid a, \\
	\frac{p}{p^2-1}\Big(1-\frac{\lambda_{\phi}(p^2)}{p}+\frac{\lambda_{\phi}(p^2)}{p^2}- \frac{1}{p^3}\Big) \quad &\text{if } p\nmid a.
	\end{cases}
	\]
	By the estimate \eqref{kimsarnak-bound}, one has
	\[
	g(p)=\begin{cases}
	\frac{\lambda_{\phi}(p)^2}{p}+O\big(p^{-\frac{39}{32}+\varepsilon}\big)  \quad &\text{if } p\mid a, \\
	\frac{1}{p}+O\big(p^{-\frac{39}{32}+\varepsilon}\big) \quad &\text{if } p\nmid a.
	\end{cases}
	\]
	Together this with Mertens' theorem or prime number theorem of Rankin--Selberg $L$-function $L(s,{\phi}\otimes {\phi})$, we have
	\begin{equation}\label{g}
	\prod_{w\leq p<z}(1-g(p))^{-1} \sim  \frac{\log z}{\log w}.
	\end{equation}
	This means that the sieve dimension of the function $\lambda_{\phi}(n)^2$ is one.
	
	Next, we apply the upper-bound sieve (see, for example, \cite[Corollary 6.2]{IK}). Since the sieve dimension equals one, the parameters $\beta$ and $s$ in \cite[Corollary 6.2]{IK}) can be taken $\beta=s=10$. Thus, applying Corollary \ref{rs-ap-2} and \eqref{g}, we find
	\begin{equation*}
	\begin{aligned}
	\sum_{\substack{n \leq X\\ \left(n-a, P(Y,Z)\right)=1}}\lambda_{\phi}(n)^2 \ll  & X  \prod_{p|P}(1-g(p))
	+X^{\frac{3}{5}+\varepsilon}Z^{10}\\
	\ll & X  \frac{\log Y}{\log Z}
	\end{aligned}
	\end{equation*}
	for any $Z\leq X^{\frac{1}{25}-\varepsilon}.$
\end{proof}

\section{Proof of Theorem \ref{thm-main}}

\smallskip
For convention, we write $H= (\log X)^{8c+16}$ and $K= X^{\delta/2}$,
where $c$ and $\delta>0$ are the same constants as in Hypotheses (i) and (ii). We emphasize that $\delta$ is a sufficiently small constant. Let $\nu\in[H,K]$, and set
\begin{equation*}
\I_\nu = ((\nu-1)^2,\nu^2], \quad \ P_\nu = \prod_{(H-1)^2 < p \leq \nu^2} p,
\end{equation*}
\begin{equation*}
\P_\nu = \{p\in \I_\nu\}, \quad \M_\nu = \Big\{m\in \big[1,X/\nu^2\big]: (m, P_\nu)=1 \Big\},
\end{equation*}
\vskip 0.01mm
\begin{equation*}
\P_\nu\M_\nu = \{pm: p\in\P_\nu, m\in\M_\nu\},
\end{equation*}
\begin{equation*}
\I= \bigcup_{H\leq \nu \leq K} \P_\nu\M_\nu \ \ \text{and} \ \ \J= [1,X] \setminus \I.
\end{equation*}
The intervals above are always meant as subsets of $\NN$. Note that $\P_\nu\M_\nu \subset [1,X]$ and each $n\in \P_\nu\M_\nu$ can be written in a unique way as $n=pm$ with $p\in\P_\nu$ and $m\in\M_\nu$.  Moreover, the sets $\P_\nu\M_\nu$ are pairwise disjoint for $H\leq \nu\leq K$.
First of all, we split the sum into two parts according to the variable $n\in \I$ or $n\in \J$
\begin{equation}\label{2-4}
\begin{aligned}
\sum_{n\leq X} f(n) \lambda_{\phi}(n+h) =& \sum_{n\in\I} f(n) \lambda_{\phi}(n+h) + \sum_{n\in\J} f(n) \lambda_{\phi}(n+h)\\
:=& S_\I(x) + S_\J(x).
\end{aligned}
\end{equation}

\vskip 2mm

\subsection{Evaluation of $S_\I(x)$}

%%%%%%%%%%%%%%%
Since $\P_\nu\M_\nu$ are pairwise disjoint for $H\leq \nu\leq K$, we obtain that
\begin{equation}
\label{2-5}
S_\I(x)\ll  \sum_{H\leq \nu\leq K}\Bigg| \sum_{pm\in\P_\nu\M_\nu} f(pm)\lambda_{\phi}(pm+h)\Bigg|.
\end{equation}
For $pm\in\P_\nu\M_\nu$, we have $p\in \P_\nu$ and $m\in\M_\nu$. The $f(pm)$ in \eqref{2-5} can be factored as $f(p)f(m)$ by its multiplicativity. Thus, we get
\begin{equation*}
\label{2-5-1}
S_\I(x) \ll \sum_{H\leq \nu\leq K}\sum_{m\in\M_\nu} |f(m)| \Big|\sum_{p\in\P_\nu} f(p)\lambda_{\phi}(pm+h)\Big|.
\end{equation*}
The inner sum could be estimated by using the Cauchy--Schwarz inequality
\begin{equation}\label{2-6}
\begin{aligned}
&\sum_{m\in\M_\nu} |f(m)|\, \Big|\sum_{p\in\P_\nu} f(p)\lambda_{\phi}(pm+h)\Big| \\
\leq &\Big(\sum_{m\in\M_\nu} |f(m)|^2\Big)^{\frac{1}{2}} \, \Big(\sum_{m\in\M_\nu} \Big|\sum_{p\in\P_\nu} f(p)\lambda_{\phi}(pm+h) \Big|^2 \Big)^{\frac{1}{2}} \\
\leq&  \Big(\sum_{m\in\M_\nu} |f(m)|^2\Big)^{\frac{1}{2}} \, \Big(\sum_{m\leq X/\nu^2} \Big|\sum_{p\in\P_\nu} f(p)\lambda_{\phi}(pm+h) \Big|^2 \Big)^{\frac{1}{2}}\\
\ll & \Big(\sum_{m\in\M_\nu} |f(m)|^2\Big)^{\frac{1}{2}}\, \Big(\sum_{p_1,p_2\in\P_\nu}|f(p)|^2 \, \Big| \sum_{m=1}^\infty \lambda_{\phi}(p_1m+h)\lambda_{\phi}(p_2m+h)e^{-\frac{\nu^2 m}{X}} \Big|\Big)^{\frac{1}{2}},
\end{aligned}
\end{equation}
where the last step inserts a redundant weight function and uses the inequality of arithmetic and geometric means.

The diagonal contribution in \eqref{2-6}, that is $p_1=p_2$ for each $\nu$ yields at most
\begin{equation}\label{2-6-1}
\begin{split}
&\Big(\sum_{m\in\M_\nu} |f(m)|^2\Big)^{\frac{1}{2}} \Big(e^{\frac{h}{X}}\sum_{p\in\P_\nu}|f(p)|^2 \sum_{\substack{n=1\\  n\equiv h\mop}}^{\infty}|\lambda_{\phi}(n)|^2e^{-\frac{n}{X}}\Big)^{\frac{1}{2}}\\
\ll  & \Big(\max_{p\in\P_\nu} \sum_{\substack{n=1\\  n\equiv h\mop}}^{\infty}|\lambda_{\phi}(n)|^2e^{-\frac{n}{X}}\Big)^{\frac{1}{2}} \Big(\sum_{p\in\P_\nu,m\in\M_\nu} |f(pm)|^2\Big)^{\frac{1}{2}} ,
\end{split}
\end{equation}
by using the multiplicative property of $f(n)$ and the fact $h\ll X$. We shall use the properties of $\lambda_{\phi}(n)$ and partial summations to estimate the first term on the second line of \eqref{2-6-1}. Corollary \ref{rs-ap-2} can be directly used. After inserting the bounds in \eqref{gdq-1}, we have
\begin{equation}\label{arithpro-pv}
\begin{aligned}
\sum_{\substack{n=1\\  n\equiv h\mop}}^{\infty}|\lambda_{\phi}(n)|^2e^{-\frac{n}{X}}\ll&\sum_{n=1}^{\infty}|\lambda_{\phi}(pn)|^2e^{-\frac{pn}{X}}+\max_{(p,h)=1} \sum_{\substack{n=1\\  n\equiv h\mop}}^{\infty}|\lambda_{\phi}(n)|^2e^{-\frac{n}{X}}\\
\ll& |\lambda_{\phi}(p)|^2\sum_{n=1}^{\infty}|\lambda_{\phi}(n)|^2e^{-\frac{pn}{X}}+\max_{(p,h)=1} \sum_{\substack{n=1\\  n\equiv h\mop}}^{\infty}|\lambda_{\phi}(n)|^2e^{-\frac{n}{X}}\\
\ll& v^{-\frac{25}{16}}X
\end{aligned}
\end{equation}
for any $p\in\P_\nu$. Hence, by summing over $\nu$ and the Cauchy--Schwarz inequality, the diagonal contribution to $S_\I(x)$ is less than
\begin{equation}\label{2-6-2}
\begin{split}
&X^{\frac{1}{2}}\Big(\sum_{H\leq \nu\leq H}v^{-\frac{25}{16}}\Big)^{\frac{1}{2}}\Big(\sum_{H\leq \nu\leq K}\sum_{p\in\P_\nu,m\in\M_\nu} |f(pm)|^2\Big)^{\frac{1}{2}}\\
&\ll \frac{X^{\frac{1}{2}}}{H^{\frac{9}{32}}}\Big(\sum_{n\leq X} |f(n)|^2\Big)^{\frac{1}{2}}\\
& \ll \frac{X(\log X)^{\frac{c}{2}}}{H^{\frac{9}{32}}}.
\end{split}
\end{equation}

Now we turn to treat the off-diagonal terms. For $p\not= q$, the innermost sum over $m$ is just a shifted convolution sum. Lemma \ref{lem-shifted-sum} can be applied and gives
\[
\sum_{m=1}^\infty \lambda_{\phi}(p_1m+h)\lambda_{\phi}(p_2m+h)e^{-\frac{\nu^2 m}{X}} \ll \nu^{\frac{7}{5}}X^{\frac{9}{10}+\varepsilon}
\]
for any $h\ll X$. Hence, by summing over $\nu$ and the Cauchy--Schwarz inequality, the contribution to $S_\I(x)$ of the off-diagonal terms in the last line of \eqref{2-6} is then at most
\begin{equation}\label{2-6-3}
\begin{aligned}
& X^{\frac{9}{20}+\varepsilon}\Big(\sum_{H\leq \nu\leq K}\nu^{\frac{9}{5}}\Big)^{\frac{1}{2}}\Big(\sum_{H\leq \nu\leq K}\sum_{p\in\P_\nu,m\in\M_\nu} |f(pm)|^2\Big)^{\frac{1}{2}}\\
& \ll K^{\frac{11}{10}}X^{\frac{9}{20}+\varepsilon}.
\end{aligned}
\end{equation}

Putting the estimates \eqref{2-6-2} and \eqref{2-6-3} together, and inserting the values of parameters $H, K$, we have
\begin{equation}
\label{2-6-4}
S_\I(x) \ll  H^{-\frac{9}{32}}X(\log X)^{\frac{c}{2}}+K^{\frac{11}{10}}X^{\frac{9}{20}+\varepsilon}\ll \frac{X}{\log X}.
\end{equation}

\vskip 2mm

\subsection{Evaluation of $S_\J(x)$}
%%%%%%%%%%%%%%%
In order to estimate the contribution of corresponding sums over the set $\J$, we define the following subsets of $[1,X]$:
\begin{equation*}
\begin{split}
\J_1^{(\nu)} &= \Big\{n\in[1,X]: n \ \text{has exactly one divisor in $\P_\nu$ and none in $\bigcup_{H\leq h < \nu}\P_h$}\Big\}, \\
\J_1 &= \bigcup_{H\leq \nu\leq K} \J_1^{(\nu)}, \\
\J_2 &= \Big\{n\in[1,X]: n \ \text{has at least one prime factor in $\bigcup_{H\leq \nu \leq K}\P_\nu$}\Big\},  \\
\J_3 &= \Big\{n\in[1,X]: n \ \text{has no prime factors in ${\bigcup_{H\leq \nu \leq K}\P_\nu}$}\Big\}.
\end{split}
\end{equation*}
By the definitions of these subsets of $[1,X]$, it is clear that $P_\nu\M_\nu \subset\J_1^{(\nu)}$. Thus we have $\I\subset\J_1$. Moreover, $\J_2\cup \J_3 =[1,X]$ and $\J_2\cap \J_3 =\emptyset$. Thus, we get
\begin{equation*}
\J \subset (\J_1\setminus \I) \cup (\J_2\setminus \J_1) \cup \J_3.
\end{equation*}
As a consequence, it follows from Hypothesis (i)  that
\begin{equation}
\label{2-9}
\begin{split}
|S_\J(x)| \ll &\sum_{n\in\J_1\setminus\I} |f(n)\lambda_{\phi}(n+h)| + \sum_{n\in\J_2\setminus\J_1} |f(n)\lambda_{\phi}(n+h)| + \sum_{n\in\J_3} |f(n)\lambda_{\phi}(n+h)| \\
\ll& X^{\frac{1}{2}}(\log X)^{\frac{c}{2}}\bigg(\sum_{n\in\J_1\setminus\I} |\lambda_{\phi}(n+h)|^2+\sum_{n\in\J_2\setminus\J_1} |\lambda_{\phi}(n+h)|^2\bigg)^{\frac{1}{2}}\\
&+ \bigg(\sum_{n\in\J_3}|f(n)|^2\Big)^{\frac{1}{2}} \bigg(\sum_{n\in\J_3}|\lambda_{\phi}(n+h)|^2\Big)^{\frac{1}{2}}.
\end{split}
\end{equation}

We first deal with the summation over $n\in \J_1\setminus\I$. It is obvious that
\[
\J_1^{(\nu)}\setminus \\P_\nu\M_\nu \subset \P_\nu \Big(\frac{x}{\nu^2},\frac{x}{(\nu-1)^2}\Big].
\]
Hence, we have
\begin{equation*}
\sum_{n\in\J_1\setminus\I} |\lambda_{\phi}(n+h)|^2 \ll \sum_{H\leq \nu\leq K} \sum_{p\in \P_\nu} \sum_{\substack{|n-X-h|\ll X/\nu \\ n\equiv h\mop}}|\lambda_{\phi}(n)|^2.
\end{equation*}
The innermost sum is actually related to the coefficients $|\lambda_{\phi}(n)|^2$ in arithmetic progression over a short interval. Notice that the length of short interval and the modulo of arithmetic progression satisfies the inequality
$
p\ll (X/\nu)^{\frac{2}{5}-\varepsilon}.
$
So we deduce from Corollary \ref{rs-ap-2} and the estimate \eqref{gdq-1}that
\[
\sum_{\substack{|n-X-h|\ll X/\nu \\  n\equiv h\mop}}|\lambda_{\phi}(n)|^2\ll \nu^{-\frac{41}{16}} X.
\]
Further, we obtain
\begin{equation}\label{2-11}
\begin{aligned}
\sum_{n\in\J_1\setminus\I} |\lambda_{\phi}(n+h)|^2\ll& X\sum_{H\leq \nu\leq K} \nu^{-\frac{25}{16}}\\
\ll& H^{-\frac{9}{16}}X.
\end{aligned}
\end{equation}

Next, we deal with the summation over $n\in \J_2\setminus \J_1$. On account of
\begin{equation*}
\J_2\setminus \J_1 \subset \bigcup_{H\leq \nu\leq K} \{n\in[1,X]: n \ \text{has at least two prime factors in $\P_\nu$}\},
\end{equation*}
we get
\begin{equation*}
\sum_{n\in\J_2\setminus\J_1} |\lambda_{\phi}(n+h)|^2 \ll  \sum_{H\leq \nu\leq K} \sum_{p_1,p_2\in\P_\nu} \sum_{\substack{n \leq X+h\\  n\equiv h\mopq}}|\lambda_{\phi}(n)|^2.
\end{equation*}
Similar to the arguments of \eqref{arithpro-pv}, we use Corollary \ref{rs-ap-2} again and then obtain
\begin{equation}\label{2-12}
\begin{aligned}
\sum_{n\in\J_2\setminus\J_1} |\lambda_{\phi}(n+h)|^2 \ll&  X  \sum_{H\leq \nu\leq K} \Big(\frac{|\P_\nu|}{\nu^{25/16}}\Big)^2\\
\ll& H^{-\frac{1}{8}}X.
\end{aligned}
\end{equation}

Finally, we deal with these two summations over $n\in \J_3$. By Lemma \ref{lem-sieve-rs} and Hypothesis (ii), we have
\begin{equation}\label{2-13}
\sum_{n\in\J_3}|\lambda_{\phi}(n+h)|^2=\sum_{\substack{n\leq X+h \\ (n-h,P)=1}}|\lambda_{\phi}(n)|^2 \ll X \frac{\log\log X}{\log X}
\end{equation}
and
\begin{equation}\label{2-14}
\sum_{n\in\J_3}|f(n)|^2=\sum_{\substack{n\leq X \\ (n,P)=1}}|f(n)|^2 \ll  X\frac{(\log\log X)^\gamma}{\log X},
\end{equation}
where the notation $P$ is given in Hypothesis (ii).

Now we are ready to estimate the sum $S_\J(x)$. On inserting the estimates \eqref{2-11}--\eqref{2-14} into \eqref{2-9}, we obtain
\begin{equation}
\label{2-13}
S_\J(x) \ll H^{-\frac{1}{16}} X (\log X)^{\frac{c}{2}} + X\frac{(\log \log X)^\frac{\gamma+1}{2}}{\log X}\ll X\frac{(\log \log X)^\frac{\gamma+1}{2}}{\log X}.
\end{equation}
Therefore, our theorem \ref{thm-main} follows from \eqref{2-4}, \eqref{2-6-4} and \eqref{2-13}.

\end{document}